\documentclass{article}
\usepackage{amssymb}
\usepackage{amsfonts}
\usepackage{amsmath}

\newtheorem{theorem}{Theorem}

\newtheorem{corollary}[theorem]{Corollary}

\newtheorem{definition}[theorem]{Definition}

\newtheorem{proposition}[theorem]{Proposition}
\newtheorem{remark}[theorem]{Remark}

\newenvironment{proof}[1][Proof]{\noindent\textbf{#1.} }{\ \rule{0.5em}{0.5em}}
\input{tcilatex}
\begin{document}

\title{Polynomials Related to Harmonic Numbers and Evaluation of Harmonic
Number Series I}
\author{Ayhan Dil and Veli Kurt \\
Department of Mathematics, Akdeniz University, 07058 Antalya Turkey\\
adil@akdeniz.edu.tr, vkurt@akdeniz.edu.tr}
\maketitle

\begin{abstract}
In this paper we focus on two new families of polynomials which are
connected with exponential polynomials $\phi _{n}\left( x\right) $ and
geometric polynomials $F_{n}\left( x\right) $. We discuss their
generalizations and show that these new families of polynomials and their
generalizations are useful to obtain closed forms of some series related to
harmonic numbers.

\textbf{2000 Mathematics Subject Classification. }11B73, 11B75, 11B83

\textbf{Key words: }Exponential numbers and polynomials, Bell numbers and
polynomials, geometric numbers and polynomials, Fubini numbers and
polynomials, harmonic and hyperharmonic numbers.
\end{abstract}

\section{Introduction}

\qquad In this work we are interested in two new families of polynomials,
namely $harmonic-$exponential polynomials and $harmonic-$geometric
polynomials. We introduce these polynomials and discuss several interesting
generalizations of them with the help of Theorem $\ref{Main Theorem}$.
Furthermore we list these polynomials and the numbers that we obtain from
these polynomials.

Suppose we are given an entire function $f$\ and a function $g$, analytic in
a region containing the annulus $K=\{z:r<|z|<R\}$ where $0<r<R$. Hence these
functions have following series expansions,%
\begin{equation*}
f\left( x\right) =\sum_{n=0}^{\infty }p_{n}x^{n}\text{ and }g\left( x\right)
=\sum_{n=-\infty }^{\infty }q_{n}x^{n}.
\end{equation*}%
The following theorem is related to the functions $f$ and $g$.

\begin{theorem}
\label{Main Theorem}$\left( \cite{B}\right) $ Let the functions $f$ and $g$
be described as above. If the series%
\begin{equation*}
\sum_{n=-\infty }^{\infty }q_{n}f\left( n\right) x^{n}
\end{equation*}%
converges absolutely on $K$, then%
\begin{equation}
\sum_{n=-\infty }^{\infty }q_{n}f\left( n\right) x^{n}=\sum_{m=0}^{\infty
}p_{m}\sum_{k=0}^{m}\QATOPD\{ \} {m}{k}x^{k}g^{\left( k\right) }\left(
x\right)  \label{0}
\end{equation}%
holds for all $x\in K$.
\end{theorem}

Throughout this paper we consider the function $g$ in Theorem $\ref{Main
Theorem}$ as a function which is analytic on the disk $K=\{z:|z|<R\}$. Hence
the formula $\left( \ref{0}\right) $\ turns out to be%
\begin{equation}
\sum_{n=0}^{\infty }\frac{g^{\left( n\right) }\left( 0\right) }{n!}f\left(
n\right) x^{n}=\sum_{n=0}^{\infty }\frac{f^{\left( n\right) }\left( 0\right) 
}{n!}\sum_{k=0}^{n}\QATOPD\{ \} {n}{k}x^{k}g^{\left( k\right) }\left(
x\right) .  \label{i0}
\end{equation}

We show that families of polynomials and their generalizations presented in
this paper are considerably useful to obtain closed forms of some series
related to harmonic numbers. For instance we obtain following closed forms:%
\begin{eqnarray}
\sum_{n=1}^{\infty }\left( \sum_{k=1}^{n}kH_{k}\right) x^{n} &=&\frac{%
x\left( 1-\ln \left( 1-x\right) \right) }{\left( 1-x\right) ^{3}}\text{,}
\label{s3} \\
\sum_{n=1}^{\infty }\left( \sum_{k=1}^{n}k^{2}H_{k}\right) x^{n} &=&\frac{%
x\left( 1+2x-\left( 1+x\right) \ln \left( 1-x\right) \right) }{\left(
1-x\right) ^{4}}\text{.}  \label{s4}
\end{eqnarray}%
Also for hyperharmonic series, one of our results is%
\begin{equation}
\sum_{n=1}^{\infty }\left( \sum_{k=1}^{n}kH_{k}^{\left( \alpha \right)
}\right) x^{n}=\frac{x\left( 1-\alpha \ln \left( 1-x\right) \right) }{\left(
1-x\right) ^{\alpha +2}}  \label{s7}
\end{equation}%
where $\alpha $ is a nonnegative integer.

In the rest of this section we will introduce some important notions.

\textbf{Stirling numbers of the first and second kind}

Stirling numbers of the first kind $\QATOPD[ ] {n}{k}$ and Stirling numbers
of the second kind $\QATOPD\{ \} {n}{k}$ are defined by means of $\left( 
\cite{AS, GKP}\right) $%
\begin{equation}
\left( x\right) _{n}=x\left( x-1\right) \ldots \left( x-n+1\right)
=\sum_{k=0}^{n}\QATOPD[ ] {n}{k}x^{k}  \label{i1}
\end{equation}%
and%
\begin{equation}
x^{n}=\sum_{k=0}^{n}\QATOPD\{ \} {n}{k}\left( x\right) _{k}  \label{i2}
\end{equation}%
respectively. These numbers are quite common in combinatorics $\left( \cite%
{BG, BQ, C, Ri}\right) $.

We note that for $n\geq k\geq 1$ the following identity holds for Stirling
numbers of the second kind%
\begin{equation}
\QATOPD\{ \} {n}{k}=\QATOPD\{ \} {n-1}{k-1}+k\QATOPD\{ \} {n-1}{k}.
\label{i3}
\end{equation}%
\qquad \qquad \qquad \qquad \qquad \qquad \qquad \qquad \qquad \qquad \qquad
\qquad \qquad \qquad \qquad \qquad \qquad \qquad \qquad \qquad

There is a certain generalization of these numbers namely $r$-Stirling
numbers $\left( \cite{Br}\right) $ which is similar to the weighted Stirling
numbers $\left( \cite{CA1, CA2}\right) $. Combinatorial meanings, recurrence
relations, generating functions and several properties of these numbers are
given in $\cite{Br}$.

\textbf{Exponential polynomials and numbers}

Exponential polynomials (or single variable Bell polynomials) $\phi
_{n}\left( x\right) $ are defined by $\left( \cite{BL1, Ri}\right) $%
\begin{equation}
\phi _{n}\left( x\right) :=\sum_{k=0}^{n}\QATOPD\{ \} {n}{k}x^{k}.
\label{i5}
\end{equation}

Grunert stated these polynomials in terms of Stirling numbers of the second
kind and obtained some fundamental formulas $\left( \cite{G}\right) $.
Besides Grunert $\left( \cite{G}\right) $, mainly Ramanujan $\left( \cite{BB}%
\right) $, Bell $\left( \cite{BL1}\right) $ and\ Touchard $\left( \cite{T}%
\right) $ are well-known studies on these polynomials. We refer $\cite{B2}$\
for comprehensive information on exponential polynomials.

The first few exponential polynomials are:%
\begin{equation}
\begin{tabular}{|l|}
\hline
$\phi _{0}\left( x\right) =1\text{,}$ \\ \hline
$\phi _{1}\left( x\right) =x\text{,}$ \\ \hline
$\phi _{2}\left( x\right) =x+x^{2}\text{,}$ \\ \hline
$\phi _{3}\left( x\right) =x+3x^{2}+x^{3}\text{,}$ \\ \hline
$\phi _{4}\left( x\right) =x+7x^{2}+6x^{3}+x^{4}\text{.}$ \\ \hline
\end{tabular}
\label{i6}
\end{equation}

The well known exponential numbers (or Bell numbers) $\left( \cite{BL2, C,
CG}\right) $\ are obtained by setting $x=1$ in $\phi _{n}\left( x\right) $,$%
\,$i.e%
\begin{equation}
\phi _{n}:=\phi _{n}\left( 1\right) =\sum_{k=0}^{n}\QATOPD\{ \} {n}{k}.
\label{i10}
\end{equation}%
The first few exponential numbers are:%
\begin{equation}
\phi _{0}=1\text{, }\phi _{1}=1\text{, }\phi _{2}=2\text{, }\phi _{3}=5\text{%
, }\phi _{4}=15\text{.}  \label{i11}
\end{equation}

\textbf{Geometric polynomials and numbers}

Geometric polynomials are defined in \cite{S, ST} as follows:%
\begin{equation}
F_{n}\left( x\right) :=\sum_{k=0}^{n}\QATOPD\{ \} {n}{k}k!x^{k}.  \label{i13}
\end{equation}

We use $F_{n}$ as one of the most common notations for these polynomials in
the honor of Guido Fubini $\left( \cite{IS}\right) $. These polynomials are
also called as Fubini polynomials $\left( \cite{B}\right) $ or ordered Bell
polynomials $\left( \cite{ST}\right) $.

The first few geometric polynomials are:%
\begin{equation}
\begin{tabular}{|l|}
\hline
$F_{0}\left( x\right) =1\text{,}$ \\ \hline
$F_{1}\left( x\right) =x\text{,}$ \\ \hline
$F_{2}\left( x\right) =x+2x^{2}\text{,}$ \\ \hline
$F_{3}\left( x\right) =x+6x^{2}+6x^{3}\text{,}$ \\ \hline
$F_{4}\left( x\right) =x+14x^{2}+36x^{3}+24x^{4}\text{.}$ \\ \hline
\end{tabular}
\label{i13+}
\end{equation}

Specializing $x=1$ in $\left( \ref{i13}\right) $ we get geometric numbers
(or ordered Bell numbers) $F_{n}$ as $\left( \cite{B, ST, W}\right) $:%
\begin{equation}
F_{n}:=F_{n}\left( 1\right) =\sum_{k=0}^{n}\QATOPD\{ \} {n}{k}k!.
\label{i14}
\end{equation}

According to $\left( \cite{IS}\right) $,\ these numbers are called Fubini
numbers by Comtet.

The first few geometric numbers are:%
\begin{equation}
F_{0}=1\text{, }F_{1}=1\text{, }F_{2}=3\text{, }F_{3}=13\text{, }F_{4}=75%
\text{.}  \label{i15}
\end{equation}

Boyadzhiev $\left( \cite{B}\right) $\ introduced\ the "general geometric
polynomials" as%
\begin{equation}
F_{n,r}\left( x\right) =\frac{1}{\Gamma \left( r\right) }\sum_{k=0}^{n}%
\QATOPD\{ \} {n}{k}\Gamma \left( k+r\right) x^{k}\text{,}  \label{i16+}
\end{equation}%
where Re$\left( r\right) >0$. In the third section we will deal with the
general geometric polynomials.

Exponential and geometric polynomials are connected by the relation $\left( 
\cite{B}\right) $%
\begin{equation}
F_{n}\left( z\right) =\int_{0}^{\infty }\phi _{n}\left( z\lambda \right)
e^{-\lambda }d\lambda .  \label{i16}
\end{equation}

In \cite{DK} the authors obtained some fundemental properties of exponential
and geometric polynomials and numbers using Euler- Seidel matrices method.

\textbf{Harmonic and Hyperharmonic numbers}

The $n$-th harmonic number is defined by the $n$-th partial sum of the
harmonic series as%
\begin{equation}
H_{n}:=\sum_{k=1}^{n}\frac{1}{k}\text{,}  \label{i17}
\end{equation}

where $H_{0}=0$.

For an integer $\alpha >1$, let

\begin{equation}
H_{n}^{(\alpha )}:=\sum_{k=1}^{n}H_{k}^{(\alpha -1)}  \label{i18}
\end{equation}%
with $H_{n}^{(1)}:=H_{n}$, be the $n$-th hyperharmonic number of order $%
\alpha $ $\left( \cite{BG, CG}\right) $.

These numbers can be expressed in terms of binomial coefficients and
ordinary harmonic numbers as $\left( \cite{CG, MD}\right) $:%
\begin{equation}
H_{n}^{(\alpha )}=\binom{n+\alpha -1}{\alpha -1}(H_{n+\alpha -1}-H_{\alpha
-1}).  \label{i20}
\end{equation}

Well-known generating functions of the harmonic and hyperharmonic numbers
are given as%
\begin{equation}
\sum_{n=1}^{\infty }H_{n}x^{n}=-\frac{\ln \left( 1-x\right) }{1-x}
\label{i21}
\end{equation}%
and%
\begin{equation}
\sum_{n=1}^{\infty }H_{n}^{\left( \alpha \right) }x^{n}=-\frac{\ln \left(
1-x\right) }{\left( 1-x\right) ^{\alpha }}  \label{i22}
\end{equation}%
respectively $\left( \cite{DM}\right) $.

The following relations connect harmonic and hyperharmonic numbers with the
Stirling and $r$-Stirling numbers of the first kind as $\left( \cite{BG}%
\right) $:%
\begin{equation}
\QATOPD[ ] {k+1}{2}=k!H_{k},  \label{i23}
\end{equation}%
and%
\begin{equation}
k!H_{k}^{\left( r\right) }=\QATOPD[ ] {n+r}{r+1}_{r}.  \label{i24}
\end{equation}

\section{Transformation of harmonic numbers}

\qquad In this section we study the series related to harmonic numbers using
the transformation formula $\left( \ref{i0}\right) $.

We set $g$ in $\left( \ref{i0}\right) $ as the generating function of
harmonic numbers which is given by equation $\left( \ref{i21}\right) $.
After rearranging the $kth$ derivative of the RHS of $\left( \ref{i21}%
\right) $ we obtain the following nice looking result:

\begin{proposition}
\label{prinduction}%
\begin{equation}
\frac{d^{k}}{d^{k}x}\left\{ -\frac{\ln \left( 1-x\right) }{1-x}\right\} =%
\frac{k!\left( H_{k}-\ln \left( 1-x\right) \right) }{\left( 1-x\right) ^{k+1}%
}.  \label{h1}
\end{equation}
\end{proposition}

\begin{proof}
Follows by induction on $k$.
\end{proof}

From $\left( \ref{h1}\right) $ we have%
\begin{equation}
g^{\left( k\right) }\left( x\right) =\frac{k!\left( H_{k}-\ln \left(
1-x\right) \right) }{\left( 1-x\right) ^{k+1}}  \label{h1+}
\end{equation}%
and%
\begin{equation}
g^{\left( k\right) }\left( 0\right) =k!H_{k}.  \label{h2}
\end{equation}

Now we are ready to state a transformation formula for the series related to
harmonic numbers.

\begin{proposition}
\label{ph}For an entire function $f$\ the following transformation formula
holds.%
\begin{eqnarray}
\sum_{n=0}^{\infty }H_{n}f\left( n\right) x^{n} &=&\frac{1}{1-x}%
\sum_{n=0}^{\infty }\frac{f^{\left( n\right) }\left( 0\right) }{n!}%
\sum_{k=0}^{n}\QATOPD\{ \} {n}{k}k!H_{k}\left( \frac{x}{1-x}\right) ^{k}
\label{h3} \\
&&-\frac{\ln \left( 1-x\right) }{1-x}\sum_{n=0}^{\infty }\frac{f^{\left(
n\right) }\left( 0\right) }{n!}\sum_{k=0}^{n}\QATOPD\{ \} {n}{k}k!\left( 
\frac{x}{1-x}\right) ^{k}.  \notag
\end{eqnarray}
\end{proposition}

\begin{proof}
Employing $\left( \ref{h1+}\right) $, $\left( \ref{h2}\right) $ in $\left( %
\ref{i0}\right) $ gives the statement.
\end{proof}

Geometric polynomials $F_{n}\left( x\right) $ appear in the second part of
the RHS of equation $\left( \ref{h3}\right) $. The first part of the RHS
contains a new family of polynomials. We will indicate them with $%
F_{n}^{h}\left( x\right) $ and call them as "$harmonic-$geometric
polynomials" because of their factor $H_{k}$. Hence the $harmonic-$geometric
polynomials are%
\begin{equation}
F_{n}^{h}\left( x\right) :=\sum_{k=0}^{n}\QATOPD\{ \} {n}{k}k!H_{k}x^{k}.
\label{h4}
\end{equation}

The first few $harmonic$-geometric polynomials are:%
\begin{equation}
\begin{tabular}{|l|}
\hline
$F_{0}^{h}\left( x\right) =0$, \\ \hline
$F_{1}^{h}\left( x\right) =x$, \\ \hline
$F_{2}^{h}\left( x\right) =3x^{2}+x$, \\ \hline
$F_{3}^{h}\left( x\right) =11x^{3}+9x^{2}+x$, \\ \hline
$F_{4}^{h}\left( x\right) =50x^{4}+66x^{3}+21x^{2}+x$, \\ \hline
$F_{5}^{h}\left( x\right) =274x^{5}+500x^{4}+275x^{3}+45x^{2}+x$. \\ \hline
\end{tabular}
\label{Lhfp}
\end{equation}

Using these notation we reformulate equation $\left( \ref{h3}\right) $ as
follows:%
\begin{equation}
\sum_{n=0}^{\infty }H_{n}f\left( n\right) x^{n}=\frac{1}{1-x}%
\sum_{n=0}^{\infty }\frac{f^{\left( n\right) }\left( 0\right) }{n!}\left\{
F_{n}^{h}\left( \frac{x}{1-x}\right) -F_{n}\left( \frac{x}{1-x}\right) \ln
\left( 1-x\right) \right\} .  \label{h5}
\end{equation}%
Formula $\left( \ref{h5}\right) $\ enables us to calculate closed forms of
some series related to harmonic numbers. Hence by means of $\left( \ref{h4}%
\right) $\ we give a corollary of Proposition $\ref{ph}$.

\begin{corollary}
For any nonnegative integer $m$ the following equality holds.%
\begin{equation}
\sum_{n=1}^{\infty }n^{m}H_{n}x^{n}=\frac{1}{1-x}\left\{ F_{m}^{h}\left( 
\frac{x}{1-x}\right) -F_{m}\left( \frac{x}{1-x}\right) \ln \left( 1-x\right)
\right\} .  \label{h6}
\end{equation}
\end{corollary}

\begin{proof}
It follows directly by setting $f\left( x\right) =x^{m}$ in $\left( \ref{h5}%
\right) $.
\end{proof}

\begin{remark}
Equation $\left( \ref{h6}\right) $ is a generalization of the generating
function of harmonic numbers, since the case $m=0$ gives equation $\left( %
\ref{i21}\right) $. Besides this ordinary case thanks to the formula $\left( %
\ref{h6}\right) $\ we obtain generating functions of several interesting
series related to harmonic numbers. For instance the case $m=1$ in\ $\left( %
\ref{h6}\right) $\ gives%
\begin{equation}
\sum_{n=1}^{\infty }nH_{n}x^{n}=\frac{x\left( 1-\ln \left( 1-x\right)
\right) }{\left( 1-x\right) ^{2}}\text{,}  \label{h7a}
\end{equation}%
and the case $m=2$ gives%
\begin{equation}
\sum_{n=1}^{\infty }n^{2}H_{n}x^{n}=\frac{x\left( 1+2x-\left( 1+x\right) \ln
\left( 1-x\right) \right) }{\left( 1-x\right) ^{3}}\text{,}  \label{h7b}
\end{equation}%
and so on.
\end{remark}

Now we extend our results to multiple sums.

\begin{proposition}
We have%
\begin{eqnarray}
&&\sum_{n=1}^{\infty }\left( \sum_{r=0}^{n}\binom{n+s-r}{s}r^{m}H_{r}\right)
x^{n}=\sum_{n=1}^{\infty }\left( \sum_{0\leq i\leq i_{1}\leq \cdots \leq
i_{s}\leq n}i^{m}H_{i}\right) x^{n}  \notag \\
&=&\frac{1}{\left( 1-x\right) ^{s+2}}\left[ F_{m}^{h}\left( \frac{x}{1-x}%
\right) -F_{m}\left( \frac{x}{1-x}\right) \ln \left( 1-x\right) \right] 
\text{,}  \label{h8}
\end{eqnarray}%
where $m$ and $s$ are nonnegative integers.
\end{proposition}

\begin{proof}
By multiplying both sides of equation $\left( \ref{h6}\right) $ with the
Newton binomial series and considering%
\begin{equation}
\sum_{r=0}^{n}\binom{n+s-r}{s}r^{m}H_{r}=\sum_{0\leq i\leq i_{1}\leq \cdots
\leq i_{s}\leq n}i^{m}H_{i}.  \label{h9}
\end{equation}%
we obtain the statement.
\end{proof}

\begin{corollary}
\begin{equation*}
\sum_{n=0}^{\infty }H_{n}^{\left( s\right) }x^{n}=\frac{-\ln \left(
1-x\right) }{\left( 1-x\right) ^{s}}.
\end{equation*}
\end{corollary}

\begin{proof}
Setting $m=0$ and considering $F_{0}^{h}\left( x\right) =0$ and $F_{0}\left(
x\right) =1$ gives the desired result.
\end{proof}

As a one more corollary we have following graceful identity.

\begin{corollary}
\begin{eqnarray}
&&\sum_{n=1}^{\infty }\left( 1^{m}H_{1}+2^{m}H_{2}+\cdots +n^{m}H_{n}\right)
x^{n}  \label{h10} \\
&=&\frac{1}{\left( 1-x\right) ^{2}}\left[ F_{m}^{h}\left( \frac{x}{1-x}%
\right) -F_{m}\left( \frac{x}{1-x}\right) \ln \left( 1-x\right) \right] . 
\notag
\end{eqnarray}
\end{corollary}

Other values of $m$ in $\left( \ref{h10}\right) $ lead the following
interesting sums:

$m=1$ in $\left( \ref{h10}\right) $\ gives%
\begin{equation}
\sum_{n=1}^{\infty }\left( \sum_{k=1}^{n}kH_{k}\right) x^{n}=\frac{x\left(
1-\ln \left( 1-x\right) \right) }{\left( 1-x\right) ^{3}}\text{,}
\label{sm1}
\end{equation}

$m=2$ gives%
\begin{equation}
\sum_{n=1}^{\infty }\left( \sum_{k=1}^{n}k^{2}H_{k}\right) x^{n}=\frac{%
x\left( 1+2x-\left( 1+x\right) \ln \left( 1-x\right) \right) }{\left(
1-x\right) ^{4}}\text{,}  \label{sm2}
\end{equation}%
and so on.

\begin{remark}
All these formulas and equations which we obtained until now, show that $%
harmonic$-geometric polynomials have strong relation with the series of
harmonic numbers. We could state the generating functions of some series
related to harmonic numbers in terms of $harmonic$-geometric polynomials,
see for instance equations $\left( \ref{h6}\right) $ and $\left( \ref{h8}%
\right) $.
\end{remark}

\begin{remark}
\label{RP}Most of the results in this section are obtained by setting $%
f\left( x\right) =x^{m}$ in the transformation formula $\left( \ref{h5}%
\right) $. It is possible to obtain more general results by setting $f\left(
x\right) $ in $\left( \ref{h5}\right) $\ as an arbitrary polynomial of order 
$m$ as%
\begin{equation}
f\left( x\right) =p_{m}x^{m}+p_{m-1}x^{m-1}+\cdots +p_{1}x+p_{0}  \label{p0}
\end{equation}%
where $p_{0},$ $p_{1},$ $\cdots ,p_{m-1},$ $p_{m}$ are any complex numbers.
Hence we get following equation which is more general than $\left( \ref{h6}%
\right) $:%
\begin{eqnarray}
&&\sum_{n=0}^{\infty }\left( p_{m}n^{m}+p_{m-1}n^{m-1}+\cdots
+p_{1}n+p_{0}\right) H_{n}x^{n}  \label{p1} \\
&=&\frac{1}{1-x}\sum_{k=0}^{m}p_{k}\left\{ F_{k}^{h}\left( \frac{x}{1-x}%
\right) -F_{k}\left( \frac{x}{1-x}\right) \ln \left( 1-x\right) \right\} . 
\notag
\end{eqnarray}
\end{remark}

Specializing coefficients of $f$ gives more closed forms of harmonic number
series. Each polynomial creates another sum. For instance by setting $%
p_{k}=1 $ for each $k=0,1,\ldots ,n$ in $\left( \ref{p0}\right) $\ we get%
\begin{eqnarray}
&&\sum_{n=0}^{\infty }\left( n^{m}+n^{m-1}+\cdots +n+1\right) H_{n}x^{n}
\label{p2} \\
&=&\frac{1}{1-x}\sum_{k=0}^{m}\left\{ F_{k}^{h}\left( \frac{x}{1-x}\right)
-F_{k}\left( \frac{x}{1-x}\right) \ln \left( 1-x\right) \right\} .  \notag
\end{eqnarray}%
This formula leads the following sums:

The case $m=1$ in $\left( \ref{p2}\right) $ gives%
\begin{equation}
\sum_{n=1}^{\infty }\left( n+1\right) H_{n}x^{n}=\frac{x-\ln \left(
1-x\right) }{\left( 1-x\right) ^{2}}.  \label{p3}
\end{equation}

The case $m=2$ in $\left( \ref{p2}\right) $ gives%
\begin{equation}
\sum_{n=1}^{\infty }\left( n^{2}+n+1\right) H_{n}x^{n}=\frac{x^{2}+2x-\left(
1+x^{2}\right) \ln \left( 1-x\right) }{\left( 1-x\right) ^{3}}\text{,}
\label{p4}
\end{equation}%
and so on.

By setting $p_{k}=k$ for each $k=0,1,\ldots ,n$ in $\left( \ref{p0}\right) $%
\ we get%
\begin{eqnarray}
&&\sum_{n=1}^{\infty }\left( mn^{m}+\left( m-1\right) n^{m-1}+\cdots
+n\right) H_{n}x^{n}  \label{p5} \\
&=&\frac{1}{1-x}\sum_{k=1}^{m}k\left\{ F_{k}^{h}\left( \frac{x}{1-x}\right)
-F_{k}\left( \frac{x}{1-x}\right) \ln \left( 1-x\right) \right\} .  \notag
\end{eqnarray}

We can give some examples of special cases of $\left( \ref{p5}\right) $\ as
well. For example the case $m=1$ in $\left( \ref{p5}\right) $\ gives the sum 
$\left( \ref{h7a}\right) $. The case $m=2$ in $\left( \ref{p5}\right) $ gives%
\begin{equation}
\sum_{n=1}^{\infty }n\left( 2n+1\right) H_{n}x^{n}=\frac{3x\left( 1+x\right)
-x\left( x+3\right) \ln \left( 1-x\right) }{\left( 1-x\right) ^{3}}\text{,}
\label{p6}
\end{equation}%
and so on.

We obtain some of these results by using operator argument in the following
subsection.

\subsection{\textbf{The operator }$\left( xD\right) $}

The operator $\left( xD\right) $ operates a function $f\left( x\right) $ as%
\begin{equation}
\left( xD\right) f\left( x\right) =xf^{\prime }\left( x\right)  \label{o1}
\end{equation}%
where $f^{\prime }$ is the first derivative of the function $f.$

For any $m$-times differentiable function $f$ we have $\left( \cite{B}%
\right) ,$%
\begin{equation}
\left( xD\right) ^{m}f\left( x\right) =\sum_{k=0}^{m}\QATOPD\{ \}
{m}{k}x^{k}f^{\left( k\right) }\left( x\right) .  \label{o2}
\end{equation}%
This fact can be easily proven with induction on $m$ using $\left( \ref{i3}%
\right) $.

We consider the generating function of the harmonic numbers in the formula $%
\left( \ref{o2}\right) .$ With the help of Proposition \ref{prinduction} we
have%
\begin{eqnarray*}
\left( xD\right) ^{m}\left( -\frac{\ln \left( 1-x\right) }{1-x}\right)
&=&\sum_{k=0}^{m}\QATOPD\{ \} {m}{k}x^{k}\frac{k!\left( H_{k}-\ln \left(
1-x\right) \right) }{\left( 1-x\right) ^{k+1}} \\
&=&\frac{1}{1-x}\left[ F_{m}^{h}\left( \frac{x}{1-x}\right) -F_{m}\left( 
\frac{x}{1-x}\right) \ln \left( 1-x\right) \right] .
\end{eqnarray*}%
On the other hand by using $\left( \ref{o1}\right) $ we have%
\begin{equation*}
\left( xD\right) ^{m}\left( \sum_{n=1}^{\infty }H_{n}x^{n}\right)
=\sum_{n=1}^{\infty }H_{n}n^{m}x^{n}.
\end{equation*}%
Combining these two results we obtain the formula $\left( \ref{h6}\right) $.

\subsection{Harmonic-geometric numbers}

\begin{definition}
$Harmonic-$geometric numbers $F_{n}^{h}$ are obtained by setting $x=1$ in $%
\left( \ref{h4}\right) $ as%
\begin{equation}
F_{n}^{h}:=F_{n}^{h}\left( 1\right) =\sum_{k=0}^{n}\QATOPD\{ \}
{n}{k}k!H_{k}.  \label{hf1}
\end{equation}
\end{definition}

The first few $harmonic$-geometric numbers are

\begin{equation}
F_{0}^{h}=0,\text{ }F_{1}^{h}=1,\text{ }F_{2}^{h}=4,\text{ }F_{3}^{h}=21,%
\text{ }F_{4}^{h}=138,\text{ }F_{5}^{h}=1095,\text{ }\ldots  \label{Lhfn}
\end{equation}

\begin{remark}
By using $\left( \ref{i23}\right) $ we can state $harmonic$-geometric
polynomials and numbers just in terms of Stirling numbers of the first and
second kind as follows%
\begin{equation}
F_{n}^{h}\left( x\right) =\sum_{k=0}^{n}\QATOPD\{ \} {n}{k}\QATOPD[ ] {k+1}{2%
}x^{k}\text{,}  \label{hf2}
\end{equation}%
\begin{equation}
F_{n}^{h}=\sum_{k=0}^{n}\QATOPD\{ \} {n}{k}\QATOPD[ ] {k+1}{2}.  \label{hf3}
\end{equation}
\end{remark}

\subsection{Harmonic-exponential polynomials and numbers}

\qquad Geometric and exponential polynomials are connected to each other via
equation $\left( \ref{i16}\right) $. Now with this motivation we define $%
harmonic-$exponential polynomials and numbers.

\begin{definition}
$Harmonic-$exponential polynomials and numbers are respectively given by the
following equations%
\begin{equation}
\phi _{n}^{h}\left( x\right) :=\sum_{k=0}^{n}\QATOPD\{ \} {n}{k}H_{k}x^{k}
\label{hb1}
\end{equation}%
and%
\begin{equation}
\phi _{n}^{h}:=\phi _{n}^{h}\left( 1\right) =\sum_{k=0}^{n}\QATOPD\{ \}
{n}{k}H_{k}\text{.}  \label{hb2}
\end{equation}
\end{definition}

The first few $harmonic-$exponential polynomials are,%
\begin{equation}
\begin{tabular}{|l|}
\hline
$\phi _{0}^{h}\left( x\right) =0$, \\ \hline
$\phi _{1}^{h}\left( x\right) =x,$ \\ \hline
$\phi _{2}^{h}\left( x\right) =x+\frac{3}{2}x^{2}$, \\ \hline
$\phi _{3}^{h}\left( x\right) =x+\frac{9}{2}x^{2}+\frac{11}{6}x^{3}$, \\ 
\hline
$\phi _{4}^{h}\left( x\right) =x+\frac{21}{2}x^{2}+11x^{3}+\frac{25}{12}%
x^{4} $, \\ \hline
$\phi _{5}^{h}\left( x\right) =x+\frac{45}{2}x^{2}+\frac{275}{6}x^{3}+\frac{%
250}{12}x^{4}+\frac{137}{60}x^{5}$. \\ \hline
\end{tabular}
\label{Lhbp}
\end{equation}

And $harmonic-$exponential numbers are,%
\begin{equation}
\phi _{0}^{h}=0,\text{ }\phi _{1}^{h}=1,\text{ }\phi _{2}^{h}=\frac{5}{2},%
\text{ }\phi _{3}^{h}=\frac{22}{3},\text{ }\phi _{4}^{h}=\frac{295}{12},%
\text{ }\phi _{5}^{h}=\frac{1849}{20},\text{ \ldots }  \label{Lhbn}
\end{equation}

We can extend the relation $\left( \ref{i16}\right) $ for harmonic types of
these polynomials as%
\begin{equation}
F_{n}^{h}\left( z\right) =\int_{0}^{\infty }\phi _{n}^{h}\left( z\lambda
\right) e^{-\lambda }d\lambda .  \label{hb3}
\end{equation}

\section{Hyperharmonic-geometric and exponential polynomials}

\qquad In this section we generalize almost all of our results which we
obtained in the previous section.

Now instead of the generating function of harmonic numbers, we consider the
generating function of the hyperharmonic numbers $\left( \ref{i22}\right) $\
in $\left( \ref{i0}\right) $. Let us take $g$ in $\left( \ref{i0}\right) $ as%
\begin{equation*}
g\left( x\right) =\sum_{n=1}^{\infty }H_{n}^{\left( \alpha \right) }x^{n}=-%
\frac{\ln \left( 1-x\right) }{\left( 1-x\right) ^{\alpha }}.
\end{equation*}

Next proposition gives a nice formula for $kth$ derivatives of $g\left(
x\right) $.

\begin{proposition}
\begin{equation}
\frac{d^{k}}{d^{k}x}\left\{ -\frac{\ln \left( 1-x\right) }{\left( 1-x\right)
^{\alpha }}\right\} =\frac{\Gamma \left( k+\alpha \right) }{\Gamma \left(
\alpha \right) }\frac{1}{\left( 1-x\right) ^{\alpha +k}}\left( H_{k+\alpha
-1}-H_{\alpha -1}-\ln \left( 1-x\right) \right) \text{.}  \label{hh1}
\end{equation}
\end{proposition}

\begin{proof}
Follows by induction on $k$.
\end{proof}

Hence we have%
\begin{equation}
g^{\left( k\right) }\left( x\right) =\frac{\Gamma \left( k+\alpha \right) }{%
\Gamma \left( \alpha \right) }\frac{1}{\left( 1-x\right) ^{\alpha +k}}\left(
H_{k+\alpha -1}-H_{\alpha -1}-\ln \left( 1-x\right) \right)  \label{hh1+}
\end{equation}%
and

\begin{equation}
g^{\left( k\right) }\left( 0\right) =\frac{\Gamma \left( k+\alpha \right) }{%
\Gamma \left( \alpha \right) }\left( H_{k+\alpha -1}-H_{\alpha -1}\right) 
\text{.}  \label{hh1++}
\end{equation}%
In the light of equation $\left( \ref{i20}\right) $ we can state $\left( \ref%
{hh1++}\right) $ simply as%
\begin{equation}
g^{\left( k\right) }\left( 0\right) =k!H_{k}^{\left( \alpha \right) }.
\label{hh2}
\end{equation}%
Now we are ready to prove the following proposition.

\begin{proposition}
\label{phht}Let an entire function $f$ be given. Then we have the following
transformation formula:%
\begin{eqnarray}
&&\sum_{n=0}^{\infty }H_{n}^{\left( \alpha \right) }f\left( n\right) x^{n} 
\notag \\
&=&\frac{1}{\left( 1-x\right) ^{\alpha }}\sum_{n=0}^{\infty }\frac{f^{\left(
n\right) }\left( 0\right) }{n!}\sum_{k=0}^{n}\QATOPD\{ \}
{n}{k}k!H_{k}^{\left( \alpha \right) }\left( \frac{x}{1-x}\right) ^{k}
\label{hh'3} \\
&&-\frac{\ln \left( 1-x\right) }{\left( 1-x\right) ^{\alpha }}%
\sum_{n=0}^{\infty }\frac{f^{\left( n\right) }\left( 0\right) }{n!}\frac{1}{%
\Gamma \left( \alpha \right) }\sum_{k=0}^{n}\QATOPD\{ \} {n}{k}\Gamma \left(
k+\alpha \right) \left( \frac{x}{1-x}\right) ^{k}  \notag
\end{eqnarray}
\end{proposition}

\begin{proof}
Invoking $\left( \ref{hh1+}\right) $ and $\left( \ref{hh2}\right) $ in $%
\left( \ref{i0}\right) $ gives the statement.
\end{proof}

The second part of the RHS of equation $\left( \ref{hh'3}\right) $ contains
the generalized geometric polynomials which are given by $\left( \ref{i16+}%
\right) $.

The first part of the RHS of equation $\left( \ref{hh'3}\right) $ also
contains a new family of polynomials which is a generalization of $\left( %
\ref{h4}\right) $. We call this new family as "$hyperharmonic-$geometric
polynomials". More clearly we give these polynomials as%
\begin{equation}
F_{n,\alpha }^{h}\left( x\right) =\sum_{k=0}^{n}\QATOPD\{ \}
{n}{k}k!H_{k}^{\left( \alpha \right) }x^{k}.  \label{hh5}
\end{equation}

The first few $hyperharmonic-$geometric polynomials are:

\begin{equation}
\begin{tabular}{|l|l|}
\hline
$F_{n,\alpha }^{h}\left( x\right) $ & $\alpha =2$ \\ \hline
$n=0$ & $0$ \\ \hline
$n=1$ & $x$ \\ \hline
$n=2$ & $x+5x^{2}$ \\ \hline
$n=3$ & $x+15x^{2}+26x^{3}$ \\ \hline
$n=4$ & $x+35x^{2}+156x^{3}+154x^{4}$ \\ \hline
$n=5$ & $x+75x^{2}+650x^{3}+1540x^{4}+1044x^{5}$ \\ \hline
\end{tabular}
\label{Lhhgp2}
\end{equation}%
and%
\begin{equation}
\begin{tabular}{|l|l|}
\hline
$F_{n,\alpha }^{h}\left( x\right) $ & $\alpha =3$ \\ \hline
$n=0$ & $0$ \\ \hline
$n=1$ & $x$ \\ \hline
$n=2$ & $x+7x^{2}$ \\ \hline
$n=3$ & $x+21x^{2}+47x^{3}$ \\ \hline
$n=4$ & $x+49x^{2}+282x^{3}+342x^{4}$ \\ \hline
$n=5$ & $x+105x^{2}+1175x^{3}+3420x^{4}+2754x^{5}$ \\ \hline
\end{tabular}
\label{Lhhgp3}
\end{equation}%
bear in mind that $F_{n,1}^{h}\left( x\right) =F_{n}^{h}\left( x\right) $
and we have a short list of these polynomials as $\left( \ref{Lhfp}\right) $.

In the case of $\alpha =1$ equation $\left( \ref{hh5}\right) $ gives
equation $\left( \ref{h4}\right) $.

With the help of these notation we can write the transformation formula $%
\left( \ref{hh'3}\right) $ simply as 
\begin{eqnarray}
&&\sum_{n=0}^{\infty }H_{n}^{\left( \alpha \right) }f\left( n\right) x^{n} 
\notag \\
&=&\frac{1}{\left( 1-x\right) ^{\alpha }}\sum_{n=0}^{\infty }\frac{f^{\left(
n\right) }\left( 0\right) }{n!}\left[ F_{n,\alpha }^{h}\left( \frac{x}{1-x}%
\right) -F_{n,\alpha }\left( \frac{x}{1-x}\right) \ln \left( 1-x\right) %
\right] .  \label{hh6}
\end{eqnarray}

Now we give a plain formula as a corollary of Proposition $\ref{phht}$.

\begin{corollary}
\begin{equation}
\sum_{n=1}^{\infty }n^{m}H_{n}^{\left( \alpha \right) }x^{n}=\frac{1}{\left(
1-x\right) ^{\alpha }}\left[ F_{m,\alpha }^{h}\left( \frac{x}{1-x}\right)
-F_{m,\alpha }\left( \frac{x}{1-x}\right) \ln \left( 1-x\right) \right]
\label{hh7}
\end{equation}%
where $m$ is a nonnegative integer.
\end{corollary}

\begin{proof}
Directly seen from the specializing $f\left( x\right) =x^{m}$ in $\left( \ref%
{hh6}\right) $.
\end{proof}

\begin{remark}
Formula $\left( \ref{hh7}\right) $\ is also a generalization of the
generating function of hyperharmonic numbers since the case $m=0$ gives $%
\left( \ref{i22}\right) $.
\end{remark}

Equation $\left( \ref{hh7}\right) $\ also makes it possible to get closed
forms of some series related to hyperharmonic numbers, for instance\ the
case $m=1$ in $\left( \ref{hh7}\right) $ gives%
\begin{equation}
\sum_{n=1}^{\infty }nH_{n}^{\left( \alpha \right) }x^{n}=\frac{x\left(
1-\alpha \ln \left( 1-x\right) \right) }{\left( 1-x\right) ^{\alpha +1}}
\label{Hh8}
\end{equation}%
where $F_{1,\alpha }\left( x\right) =\alpha x,$ $F_{1,\alpha }^{h}\left(
x\right) =x$.

Now we state a more general result which extends $\left( \ref{hh7}\right) $
to multiple sums.

\begin{proposition}
\begin{eqnarray}
&&\sum_{n=1}^{\infty }\left( \sum_{r=0}^{n}\binom{n+s-r}{s}%
r^{m}H_{r}^{\left( \alpha \right) }\right) x^{n}=\sum_{n=1}^{\infty }\left(
\sum_{0\leq i\leq i_{1}\leq \cdots \leq i_{s}\leq n}i^{m}H_{i}^{\left(
\alpha \right) }\right) x^{n}  \notag \\
&=&\frac{1}{\left( 1-x\right) ^{\alpha +s+1}}\left[ F_{m,\alpha }^{h}\left( 
\frac{x}{1-x}\right) -F_{m,\alpha }\left( \frac{x}{1-x}\right) \ln \left(
1-x\right) \right] .  \label{hh9}
\end{eqnarray}%
where $m$ and $s$ are nonnegative integers.
\end{proposition}

\begin{proof}
Multiplying both sides of equation $\left( \ref{hh7}\right) $ with the
Newton binomial series and considering%
\begin{equation}
\sum_{r=0}^{n}\binom{n+s-r}{s}r^{m}H_{r}^{\left( \alpha \right)
}=\sum_{0\leq i\leq i_{1}\leq \cdots \leq i_{s}\leq n}i^{m}H_{i}^{\left(
\alpha \right) }  \label{hh10}
\end{equation}%
completes the proof.
\end{proof}

\begin{corollary}
\begin{equation*}
\sum_{n=0}^{\infty }H_{n}^{\left( s\right) }x^{n}=-\frac{\ln \left(
1-x\right) }{\left( 1-x\right) ^{s}}
\end{equation*}
\end{corollary}

\begin{proof}
Specializing $m=0$ and considering $F_{0,\alpha }^{h}\left( x\right) =0$ and 
$F_{0,\alpha }\left( x\right) =1$ gives statement.
\end{proof}

For $m=1$ we get the following corollary.

\begin{corollary}
\begin{eqnarray}
&&\sum_{n=1}^{\infty }\left( \sum_{r=0}^{n}\binom{n+s-r}{s}rH_{r}^{\left(
\alpha \right) }\right) x^{n}  \label{hh10+} \\
&=&\sum_{n=1}^{\infty }\left( \sum_{0\leq i\leq i_{1}\leq \cdots \leq
i_{s}\leq n}iH_{i}^{\left( \alpha \right) }\right) x^{n}=\frac{x\left(
1-\alpha \ln \left( 1-x\right) \right) }{\left( 1-x\right) ^{\alpha +s+2}}. 
\notag
\end{eqnarray}
\end{corollary}

An example to the case $s=0$ is%
\begin{equation}
\sum_{n=1}^{\infty }\left( \sum_{k=1}^{n}kH_{k}^{\left( \alpha \right)
}\right) x^{n}=\frac{x\left( 1-\alpha \ln \left( 1-x\right) \right) }{\left(
1-x\right) ^{\alpha +2}}.  \label{hh11}
\end{equation}

The following gives a nice formula.

\begin{corollary}
\begin{eqnarray}
&&\sum_{n=1}^{\infty }\left( 1^{m}H_{1}^{\left( \alpha \right)
}+2^{m}H_{2}^{\left( \alpha \right) }+\cdots +n^{m}H_{n}^{\left( \alpha
\right) }\right) x^{n}  \notag \\
&=&\frac{1}{\left( 1-x\right) ^{\alpha +1}}\left[ F_{m,\alpha }^{h}\left( 
\frac{x}{1-x}\right) -F_{m,\alpha }\left( \frac{x}{1-x}\right) \ln \left(
1-x\right) \right] .  \label{hh12}
\end{eqnarray}
\end{corollary}

\begin{remark}
If we set $f\left( x\right) $\ as an arbitrary polynomial of order $m$, such
as%
\begin{equation}
f\left( x\right) =p_{m}x^{m}+p_{m-1}x^{m-1}+\cdots +p_{1}x+p_{0}
\label{hh13}
\end{equation}%
where $p_{0},$ $p_{1},$ $\cdots ,p_{m-1},$ $p_{m}$ are any complex numbers,\
instead of $f\left( x\right) =x^{m}$ in $\left( \ref{hh6}\right) $ we obtain
following general formula:%
\begin{eqnarray}
&&\sum_{n=0}^{\infty }\left( p_{m}n^{m}+p_{m-1}n^{m-1}+\cdots
+p_{1}n+p_{0}\right) H_{n}^{\left( \alpha \right) }x^{n}  \label{hhex} \\
&=&\frac{1}{\left( 1-x\right) ^{\alpha }}\sum_{k=o}^{m}p_{k}\left\{
F_{k}^{h}\left( \frac{x}{1-x}\right) -F_{k}\left( \frac{x}{1-x}\right) \ln
\left( 1-x\right) \right\} .  \notag
\end{eqnarray}%
Specializing $\left( \ref{hh13}\right) $ one can obtain several closed forms
of hyperharmonic number series in a similar fashion to what we did after
Remark \ref{RP}.
\end{remark}

\subsection{Some results using the operator $\left( xD\right) $}

\qquad For the completeness of this work we also consider these generalized
arguments using operator $\left( xD\right) $.

If we set $g$ as the generating function of hyperharmonic numbers $\left( %
\ref{i22}\right) $\ in $\left( \ref{o2}\right) $, then we get%
\begin{equation*}
\left( xD\right) ^{m}\left( -\frac{\ln \left( 1-x\right) }{\left( 1-x\right)
^{\alpha }}\right) =\frac{1}{\left( 1-x\right) ^{\alpha }}\left[ F_{n,\alpha
}^{h}\left( \frac{x}{1-x}\right) -F_{n,\alpha }\left( \frac{x}{1-x}\right)
\ln \left( 1-x\right) \right] .
\end{equation*}%
On the other hand using $\left( \ref{o1}\right) $\ we have%
\begin{equation*}
\left( xD\right) ^{m}\left( \sum_{n=1}^{\infty }H_{n}^{\left( \alpha \right)
}x^{n}\right) =\sum_{n=1}^{\infty }H_{n}^{\left( \alpha \right) }n^{m}x^{n}.
\end{equation*}%
Collecting these two results gives again equation $\left( \ref{hh7}\right) $%
\begin{equation*}
\sum_{n=1}^{\infty }H_{n}^{\left( \alpha \right) }n^{m}x^{n}=\frac{1}{\left(
1-x\right) ^{\alpha }}\left[ F_{n,\alpha }^{h}\left( \frac{x}{1-x}\right)
-F_{n,\alpha }\left( \frac{x}{1-x}\right) \ln \left( 1-x\right) \right] .
\end{equation*}

\subsection{Hyperharmonic-geometric numbers}

\begin{definition}
$Hyperharmonic-$geometric numbers $F_{n,\alpha }^{h}$ are obtained by
setting $x=1$ in $\left( \ref{hh5}\right) $ as%
\begin{equation}
F_{n,\alpha }^{h}:=F_{n,\alpha }^{h}\left( 1\right) =\sum_{k=0}^{n}\QATOPD\{
\} {n}{k}k!H_{k}^{\left( \alpha \right) }.  \label{hh14}
\end{equation}
\end{definition}

The first few $hyperharmonic-$geometric numbers are:

\begin{equation}
\begin{tabular}{|l|l|l|}
\hline
$F_{n,\alpha }^{h}$ & $\alpha =2$ & $\alpha =3$ \\ \hline
$n=0$ & $0$ & $0$ \\ \hline
$n=1$ & $1$ & $1$ \\ \hline
$n=2$ & $6$ & $8$ \\ \hline
$n=3$ & $42$ & $69$ \\ \hline
$n=4$ & $346$ & $674$ \\ \hline
$n=5$ & $3310$ & $7455$ \\ \hline
\end{tabular}
\label{Lhhgn}
\end{equation}%
bear in mind that $F_{n,1}^{h}=F_{n}^{h}$.

\begin{remark}
Using $\left( \ref{i24}\right) $\ which is a relation between hyperharmonic
numbers and $r-$Stirling numbers of the first kind, we can state $%
hyperharmonic-$geometric polynomials and numbers in terms of Stirling
numbers as%
\begin{equation}
F_{n,r}^{h}\left( x\right) =\sum_{k=0}^{n}\QATOPD\{ \} {n}{k}\QATOPD[ ] {n+r%
}{r+1}_{r}x^{k},  \label{hh15}
\end{equation}%
and%
\begin{equation}
F_{n,r}^{h}=\sum_{k=0}^{n}\QATOPD\{ \} {n}{k}\QATOPD[ ] {n+r}{r+1}_{r}
\label{hh16}
\end{equation}%
respectively. One can easily see that the relations $\left( \ref{hh15}%
\right) $ and $\left( \ref{hh16}\right) $ are the generalizations of the
relations $\left( \ref{hf2}\right) $ and $\left( \ref{hf3}\right) .$
\end{remark}

Let us continue our work by defining a generalization of the exponential
polynomials.

\subsection{Hyperharmonic-exponential polynomials and numbers}

\begin{definition}
$Hyperharmonic-$exponential polynomials and numbers are defined respectively
as%
\begin{equation}
\phi _{n,\alpha }^{h}\left( x\right) :=\sum_{k=0}^{n}\QATOPD\{ \}
{n}{k}H_{k}^{\left( \alpha \right) }x^{k}\text{,}  \label{hh17}
\end{equation}%
and%
\begin{equation}
\phi _{n,\alpha }^{h}:=\phi _{n,\alpha }^{h}\left( 1\right)
=\sum_{k=0}^{n}\QATOPD\{ \} {n}{k}H_{k}^{\left( \alpha \right) }\text{.}
\label{hh18}
\end{equation}
\end{definition}

The first few $hyperharmonic-$exponential polynomials are%
\begin{equation}
\begin{tabular}{|l|l|}
\hline
$\phi _{n,\alpha }^{h}\left( x\right) $ & $\alpha =2$ \\ \hline
$n=0$ & $0$ \\ \hline
$n=1$ & $x$ \\ \hline
$n=2$ & $x+\frac{5}{2}x^{2}$ \\ \hline
$n=3$ & $x+\frac{15}{2}x^{2}+\frac{13}{3}x^{3}$ \\ \hline
$n=4$ & $x+\frac{35}{2}x^{2}+26x^{3}+\frac{77}{12}x^{4}$ \\ \hline
$n=5$ & $x+\frac{75}{2}x^{2}+\frac{325}{3}x^{3}+\frac{385}{6}x^{4}+\frac{87}{%
10}x^{5}$ \\ \hline
\end{tabular}
\label{Lhhep2}
\end{equation}%
and%
\begin{equation}
\begin{tabular}{|l|l|}
\hline
$\phi _{n,\alpha }^{h}\left( x\right) $ & $\alpha =3$ \\ \hline
$n=0$ & $0$ \\ \hline
$n=1$ & $x$ \\ \hline
$n=2$ & $x+\frac{7}{2}x^{2}$ \\ \hline
$n=3$ & $x+\frac{21}{2}x^{2}+\frac{47}{6}x^{3}$ \\ \hline
$n=4$ & $x+\frac{49}{2}x^{2}+47x^{3}+\frac{57}{4}x^{4}$ \\ \hline
$n=5$ & $x+\frac{105}{2}x^{2}+\frac{1175}{6}x^{3}+\frac{285}{2}x^{4}+\frac{%
459}{20}x^{5}$ \\ \hline
\end{tabular}
\label{Lhhep3}
\end{equation}

And also the first few $hyperharmonic-$exponential numbers are%
\begin{equation}
\begin{tabular}{|l|l|l|}
\hline
$\phi _{n,\alpha }^{h}\left( x\right) $ & $\alpha =2$ & $\alpha =3$ \\ \hline
$n=0$ & $0$ & $0$ \\ \hline
$n=1$ & $1$ & $1$ \\ \hline
$n=2$ & $\frac{7}{2}$ & $\frac{9}{2}$ \\ \hline
$n=3$ & $\frac{77}{6}$ & $\frac{58}{3}$ \\ \hline
$n=4$ & $\frac{611}{12}$ & $\frac{347}{4}$ \\ \hline
$n=5$ & $\frac{2197}{10}$ & $\frac{24887}{60}$ \\ \hline
\end{tabular}
\label{Lhhen}
\end{equation}

Bearing in mind that, the case $\alpha =1$ gives $\phi _{n,1}^{h}\left(
x\right) =\phi _{n}^{h}\left( x\right) $ and $\phi _{n,1}^{h}=\phi _{n}^{h}$%
. The lists of $\phi _{n}^{h}\left( x\right) $ and $\phi _{n}^{h}$ in
already given in $\left( \ref{Lhbp}\right) $ and $\left( \ref{Lhbn}\right) $.

Also for these new concepts we can generalize the relation $\left( \ref{hb3}%
\right) $ as%
\begin{equation}
F_{n,\alpha }^{h}\left( z\right) =\int_{0}^{\infty }\phi _{n,\alpha
}^{h}\left( z\lambda \right) e^{-\lambda }d\lambda .  \label{hh19}
\end{equation}

\textbf{Acknowledgement}

Research of the first and second author are supported by Akdeniz University
Scientific Research project Unit.

The authors would like to thank to Professor Khristo N. Boyadzhiev for his
help and support.

\end{document}